\documentclass[11pt]{amsart}

\usepackage[T1]{fontenc}
\usepackage{amsmath,amssymb,amsthm,mathtools}
\usepackage[colorlinks=true,citecolor=blue,linkcolor=blue,urlcolor=blue]{hyperref}

\newtheorem{theorem}{Theorem}
\newtheorem{lemma}{Lemma}
\newtheorem{remark}{Remark}

\DeclareMathOperator{\Ann}{Ann}

\DeclareMathOperator{\Hilb}{Hilb}
\DeclareMathOperator{\edim}{edim}
\DeclareMathOperator{\Apolar}{Apolar}

\DeclareMathOperator{\Gr}{Gr}

\newcommand{\A}{\mathbb A}
\newcommand{\m}{\mathfrak m}
\newcommand{\kk}{k}

\title[Existence with smoothable $Q(0)$]
{Existence of a nonsmoothable local Gorenstein algebra with smoothable $Q(0)$}

\author{Ruoyu Wu}
\address{}
\email{}

\subjclass[2010]{Primary 13E10, 14C05; Secondary 13H10}
\keywords{Artinian Gorenstein algebra, Hilbert scheme of points, smoothable algebra, Macaulay inverse systems, symmetric decomposition}

\begin{document}

\begin{abstract}
We prove that there exists a local Artinian Gorenstein algebra $A$ which is not smoothable, although the first symmetric quotient $Q_A(0)$ in the symmetric decomposition of the associated graded algebra is smoothable.  The proof uses divided-power inverse systems and gives such algebras of length $31$ and embedding dimension $14$ over every algebraically closed field.
\end{abstract}

\maketitle

\section{Introduction}

Let $A$ be a local Artinian Gorenstein algebra over an algebraically closed field.  Iarrobino's symmetric decomposition writes the Hilbert function of the associated graded algebra of $A$ as a sum of symmetric Hilbert functions
\[
H(A)=\sum_{a\geq 0} H(Q_A(a)),
\]
where the $Q_A(a)$ are reflexive subquotients of $\Gr_{\mathfrak m_A}(A)$.  More precisely, in the notation of \cite{IarrobinoMaciasMarques}, $Q_A(a)=C_A(a)/C_A(a+1)$.  If $f=f_j+f_{j-1}+\cdots$ is a Macaulay dual generator of $A$, with $f_i$ homogeneous of degree $i$, then $Q_A(0)$ is a graded Artinian Gorenstein algebra whose dual generator is $f_j$; see \cite[Theorem~1.4]{IarrobinoMaciasMarques}.

In socle degree three, smoothability of this first symmetric quotient is strong enough to imply smoothability of the local algebra itself; this follows from the structure results of Casnati, Elias, Notari, and Rossi for Gorenstein local algebras with $\mathfrak m^4=0$ \cite{CasnatiEliasNotariRossi}.  One of the open problems in \cite[\S 2.5]{IarrobinoMaciasMarques} asks whether such behavior persists in higher socle degree, namely for a nonsmoothable local Gorenstein algebra $A$ such that $Q_A(0)$ is smoothable.  This note gives an affirmative answer, in the sense of an existence result by dimension count, and shows that the implication from the smoothability of $Q_A(0)$ to the smoothability of $A$ already fails in socle degree four.  The proof considers dual generators of the form
\[
F=Z^{[4]}+G(X_1,\ldots,X_{13}),
\]
where $G$ is a general cubic.  The top-degree term forces $Q_A(0)$ to be the curvilinear algebra $\kk[[z]]/(z^5)$, while a Hilbert-scheme dimension count shows that a general member of the family is nonsmoothable.  We do not exhibit a specific cubic; the result is nonconstructive in the final nonsmoothability step.  For related work on smoothability and Gorenstein loci in Hilbert schemes of points, see for instance \cite{CasnatiJelisiejewNotari}.

\begin{theorem}\label{thm:main}
Let $\kk$ be an algebraically closed field.  There exists a local Artinian Gorenstein $\kk$-algebra $A$ of length $31$ and embedding dimension $14$ such that $A$ is not smoothable, while $Q_A(0)$ is smoothable.
\end{theorem}

\section{Inverse systems}

Let
\[
S=\kk[[z,x_1,\ldots,x_{13}]]
\]
with maximal ideal $\m=(z,x_1,\ldots,x_{13})$.  Let
\[
D=\kk_{\mathrm{DP}}[Z,X_1,\ldots,X_{13}]
\]
be the divided-power dual.  Thus $z$ acts on $D$ by contraction with $Z$, and $x_i$ acts by contraction with $X_i$:
\[
z\circ Z^{[a]}=Z^{[a-1]},\qquad
x_i\circ X_i^{[a]}=X_i^{[a-1]},
\]
with the usual convention that divided powers of negative degree are zero.  In characteristic zero one may identify divided powers with ordinary powers up to factorials.

For $F\in D$, set
\[
A_F=S/\Ann_S(F).
\]
By Macaulay duality \cite{IarrobinoKanev}, $A_F$ is Artinian Gorenstein whenever the cyclic inverse system $S\circ F$ is finite-dimensional, and
\[
\operatorname{length}(A_F)=\dim_\kk(S\circ F).
\]

\section{The family}

Let $G\in \kk_{\mathrm{DP}}[X_1,\ldots,X_{13}]_3$ be a general divided-power cubic and set
\[
F=Z^{[4]}+G.
\]
We write $A=A_F$.

\begin{lemma}\label{lem:length}
Suppose that the two catalecticant maps below have rank $13$.  Then the algebra $A$ has length $31$ and embedding dimension $14$.
\end{lemma}

\begin{proof}
Consider the two catalecticant maps
\[
\kk\{x_1,\ldots,x_{13}\}\longrightarrow
\kk_{\mathrm{DP}}[X_1,\ldots,X_{13}]_2,\qquad
\theta\longmapsto \theta\circ G,
\]
and
\[
\kk\{x_ix_j:1\leq i\leq j\leq 13\}\longrightarrow
\kk_{\mathrm{DP}}[X_1,\ldots,X_{13}]_1,\qquad
\Theta\longmapsto \Theta\circ G.
\]
The condition that these maps have rank $13$ is open and nonempty.  For example, it is satisfied by
\[
G_0=X_1^{[3]}+\cdots+X_{13}^{[3]}.
\]
Hence it is satisfied by a general $G$.

For such a $G$, the cyclic inverse system $S\circ F$ is spanned by
\[
F,\quad Z^{[3]},\quad x_i\circ G,\quad
Z^{[2]},\quad x_ix_j\circ G,\quad
Z,\quad 1.
\]
These vectors are independent, except for the evident fact that the second derivatives $x_ix_j\circ G$ span a $13$-dimensional space and the third derivatives of $G$ span the one-dimensional space $\kk\cdot 1$.  Thus
\[
\dim_\kk(S\circ F)
=1+(1+13)+(1+13)+1+1
=31.
\]
Therefore $\operatorname{length}(A)=31$.

It remains to compute the embedding dimension.  If a linear form
\[
\ell=a z+\sum_{i=1}^{13} b_i x_i
\]
annihilates $F$, then
\[
\ell\circ F=aZ^{[3]}+\sum_{i=1}^{13} b_i(x_i\circ G)=0.
\]
The term $Z^{[3]}$ is independent from the $X$-quadrics $x_i\circ G$, and those $13$ quadrics are linearly independent by the first rank condition.  Hence $a=b_1=\cdots=b_{13}=0$.  Thus $\Ann_S(F)$ contains no nonzero linear form, and $\edim(A)=14$.
\end{proof}

\begin{lemma}\label{lem:q0}
For general $G$ as above, the first symmetric quotient $Q_A(0)$ is smoothable.
\end{lemma}

\begin{proof}
The socle degree of $A$ is $4$, and the highest homogeneous part of the dual generator $F$ is $F_4=Z^{[4]}$.  By the inverse-system description of the first symmetric quotient in \cite[Theorem~1.4]{IarrobinoMaciasMarques}, the graded dual of $Q_A(0)=C_A(0)/C_A(1)$ is the cyclic module generated by $F_4$.  Equivalently,
\[
Q_A(0)\cong \Apolar(Z^{[4]})
\cong \kk[[z,x_1,\ldots,x_{13}]]/(x_1,\ldots,x_{13},z^5)
\cong \kk[[z]]/(z^5).
\]
This algebra is curvilinear, hence smoothable.
\end{proof}

\section{Nonsmoothability}

Let $U$ be the nonempty open set of cubics satisfying the two rank conditions in the proof of Lemma \ref{lem:length}.  Then
\[
\dim U=\dim_\kk \kk_{\mathrm{DP}}[X_1,\ldots,X_{13}]_3
=\binom{15}{3}
=455.
\]
Put $P=\kk[z,x_1,\ldots,x_{13}]$ and $\mathfrak n=(z,x_1,\ldots,x_{13})$.

\begin{lemma}\label{lem:hilbertmorphism}
The assignment given by
\[
\Phi:U\longrightarrow \Hilb^{31}(\A^{14}),
\qquad
G\longmapsto \Ann_P(Z^{[4]}+G)
\]
is a morphism.
\end{lemma}

\begin{proof}
Since all dual generators under consideration have degree at most $4$, each annihilator contains $\mathfrak n^5$.  It is therefore enough to work in the finite-dimensional algebra $P/\mathfrak n^5$.

Let $\mathcal G$ be the universal cubic over $U$ and set $\mathcal F=Z^{[4]}+\mathcal G$.  Contraction with $\mathcal F$ gives an $\mathcal O_U$-linear map of vector bundles
\[
\mu:\mathcal O_U\otimes_\kk P/\mathfrak n^5
\longrightarrow
\mathcal O_U\otimes_\kk D_{\leq 4},
\qquad
p\longmapsto p\circ \mathcal F.
\]
By Lemma \ref{lem:length}, for every $G\in U$ one has
\[
\dim_\kk P\circ (Z^{[4]}+G)=31.
\]
Thus $\mu$ has constant rank $31$ on $U$.  Consequently $\ker\mu$ is a subbundle of $\mathcal O_U\otimes_\kk P/\mathfrak n^5$.  Moreover $\ker\mu$ is closed under multiplication by sections of $\mathcal O_U\otimes_\kk P/\mathfrak n^5$: if $p\circ\mathcal F=0$, then for every $q$ one has
\[
(qp)\circ\mathcal F=q\circ(p\circ\mathcal F)=0.
\]
Thus $\ker\mu$ is an ideal subbundle, and
\[
\mathcal B=(\mathcal O_U\otimes_\kk P/\mathfrak n^5)/\ker\mu
\]
is a locally free $\mathcal O_U$-algebra of rank $31$.  This flat family of length-$31$ quotient algebras of $P$ gives the desired morphism to the Hilbert scheme.
\end{proof}

\begin{lemma}\label{lem:injective}
The morphism $\Phi$ is injective.  In particular, its image has dimension $455$.
\end{lemma}

\begin{proof}
Suppose that
\[
\Ann_P(Z^{[4]}+G)=\Ann_P(Z^{[4]}+G').
\]
By Macaulay duality, for any dual generator $F$ the cyclic inverse system $P\circ F$ is the Matlis dual of $P/\Ann_P(F)$ inside the divided-power module; equivalently, the annihilator determines the cyclic inverse system \cite[Chapter~1]{IarrobinoKanev}.  Hence the corresponding inverse systems are equal:
\[
P\circ (Z^{[4]}+G)=P\circ (Z^{[4]}+G').
\]
In particular, $Z^{[4]}+G'$ belongs to $P\circ (Z^{[4]}+G)$, so there is some $h\in P$ such that
\[
h\circ (Z^{[4]}+G)=Z^{[4]}+G'.
\]
Write $h=c+h_1+h_{\geq 2}$ according to degree, and write the linear part as
\[
h_1=a z+\sum_{i=1}^{13} b_i x_i.
\]
Taking the homogeneous degree-four part gives $c=1$.  The homogeneous degree-three part of $h\circ (Z^{[4]}+G)$ is then
\[
G+aZ^{[3]};
\]
indeed, the constant term of $h$ contributes $G$, the term $az$ contributes $aZ^{[3]}$, the terms $b_i x_i$ contribute only quadrics from $G$, and $h_{\geq 2}$ contributes terms of degree at most $2$.  Comparing degree-three terms gives
\[
G'=G+aZ^{[3]},
\]
and since both $G$ and $G'$ involve only the variables $X_1,\ldots,X_{13}$, we get $a=0$, and therefore $G'=G$.  Thus $\Phi$ is injective.

The assertion about the dimension of the image follows from the fiber dimension theorem.
\end{proof}

\begin{proof}[Proof of Theorem \ref{thm:main}]
By Lemmas \ref{lem:length} and \ref{lem:q0}, the algebras defined above have length $31$, embedding dimension $14$, and smoothable first symmetric quotient $Q_A(0)$.

It remains to show that a general member is not smoothable.  Let $\mathcal R\subset \Hilb^{31}(\A^{14})$ be the open locus parametrizing reduced subschemes, or equivalently unordered configurations of $31$ distinct points in $\A^{14}$.  This locus is irreducible of dimension
\[
31\cdot 14=434.
\]
Set
\[
\mathcal H_{\mathrm{sm}}=\overline{\mathcal R}
\]
inside $\Hilb^{31}(\A^{14})$.  Thus $\mathcal H_{\mathrm{sm}}$ is a closed irreducible subset of dimension $434$, and a length-$31$ subscheme of $\A^{14}$ is smoothable precisely when its Hilbert-scheme point lies in $\mathcal H_{\mathrm{sm}}$.

By Lemma \ref{lem:hilbertmorphism}, $\Phi$ is a morphism, and by Lemma \ref{lem:injective} its image has dimension $455$.  Since $U$ is irreducible, the Zariski closure $\overline{\Phi(U)}$ is irreducible of dimension $455$.  If every point of $\Phi(U)$ were smoothable, then $\Phi(U)\subset \mathcal H_{\mathrm{sm}}$, and hence $\overline{\Phi(U)}\subset \mathcal H_{\mathrm{sm}}$, contradicting
\[
455>434,
\]
because a closed subset of $\mathcal H_{\mathrm{sm}}$ has dimension at most $434$.  Therefore
\[
\Phi^{-1}(\mathcal H_{\mathrm{sm}})
\]
is a proper closed subset of $U$.  For a general $G\in U$, the algebra
\[
A=\kk[[z,x_1,\ldots,x_{13}]]/\Ann_S(Z^{[4]}+G)
\]
is nonsmoothable, while $Q_A(0)\cong \kk[[z]]/(z^5)$ is smoothable.
\end{proof}

\begin{remark}
The use of divided powers is only to avoid restrictions on the characteristic of $\kk$.  In characteristic zero, the same family may be written with ordinary polynomials after the usual normalization of powers.
\end{remark}

\begin{remark}
The proof gives an existence result by dimension count.  It does not claim that length $31$ is minimal, nor does it single out a preferred explicit cubic $G$.
\end{remark}

\bibliographystyle{alpha}
\bibliography{nonsmoothable_Q0_smoothable}

\end{document}